\documentclass[final,1p,times]{elsarticle}
\usepackage{amsmath}
\allowdisplaybreaks[4]
\usepackage{amsthm}
\usepackage{amscd}
\usepackage{amsfonts}
\usepackage{amssymb}
\usepackage{graphicx}
\usepackage[pagebackref, colorlinks, citecolor=green, linkcolor=red, urlcolor=blue]{hyperref}
\newtheorem{theorem}{Theorem}[section]

\newtheorem{lemma}[theorem]{Lemma}

\usepackage{mathrsfs}
\usepackage{titletoc}
\numberwithin{equation}{section}

\newcommand{\f}{\frac}

\renewcommand{\l}{\lambda}

\newcommand{\be}{\begin{equation}}

\newcommand{\ee}{\end{equation}}
\newcommand{\bea}{\begin{eqnarray}}
\newcommand{\eea}{\end{eqnarray}}
\newcommand{\bna}{\begin{eqnarray*}}
\newcommand{\ena}{\end{eqnarray*}}

\journal{***}
\begin{document}
   
\begin{frontmatter}

\title{Solutions to  a generalized Chern-Simons Higgs model   on finite graphs by topological degree}

\author{Songbo Hou}
\ead{housb@cau.edu.cn}
\address{Department of Applied Mathematics, College of Science, China Agricultural University,  Beijing, 100083, P.R. China}
\author{Wenjie Qiao }
\ead{2228105725@qq.com}
\address{Department of Applied Mathematics, College of Science, China Agricultural University,  Beijing, 100083, P.R. China}
\cortext[cor1]{Corresponding author: Songbo Hou}

\begin{abstract}
Consider a finite connected graph denoted as $G=(V, E)$. This study explores a generalized Chern-Simons Higgs model, characterized by the equation:
$$
\Delta u = \lambda e^u (e^u - 1)^{2p+1} + f,$$
where $\Delta$ denotes the graph Laplacian, $\lambda$ is a real number, \(p\) is a non-negative integer, and \(f\) is a function on \(V\).
Through the computation of the topological degree, this paper demonstrates the existence of a single solution for the model. Further analysis of the interplay between the topological degree and the critical group of an associated functional reveals the presence of multiple solutions. These findings  extend the work of Li, Sun, Yang (arXiv:2309.12024) and Chao, Hou (J. Math. Anal. Appl. (2023) 126787).

\end{abstract}
\end{frontmatter}
\section{Introduction}

The Chern-Simons Higgs model, a significant area of research in theoretical physics, involves the integration of the Chern-Simons term within the framework of gauge theory.  Since the identification of the self-dual structure within the Abelian Chern-Simons model \cite{MR1050529,MR1050530} in 1990, there has been a surge of prolific research on self-dual Chern-Simons equations. Specifically, both topological and non-topological solutions have garnered significant attention  \cite{MR1429096,MR1690951,MR1946331,MR2168004,huang2014uniqueness,MR3390935,MR3412153}.

In \cite{MR1324400}, Caffarelli and Yang investigated the equation 
\begin{equation}\label{cye}
\Delta u=\lambda e^u\left(e^u-1\right)+4 \pi \sum_{j=1}^N \delta_{p_j} \quad \text { in } \quad \Omega,
\end{equation}
where $\Omega$ denotes the  doubly
periodic region of $
\mathbb{R}^2
$  or the  $
\text { the 2-torus } \Omega=\mathbb{R}^2 / \Omega 
$, $\lambda$ denotes a positive constant and $\delta_p$ denotes the Dirac distribution concentrated at $p \in \Omega$.  They proved that there exists a critical value of \( \lambda \), denoted as \( \lambda_c \), 
such that equation (\ref{cye}) admits a solution for \( \lambda > \lambda_c \). Conversely, for \( \lambda < \lambda_c \), the equation does not possess any solution.  Tarantello \cite{MR1400816} established the existence of solutions  for equation (\ref{cye}) if $\l=\l_c$ and obtained multiple condensate solutions if $\l>\l_c$.

Han \cite{MR3033571} established  the existence of multi-vortices for a generalized self-dual
Chern–Simons model over a   doubly
periodic region of $
\Omega\subset\mathbb{R}^2$:
\begin{equation}\label{mef}
\Delta u=\lambda e^u\left(e^u-1\right)^5+4 \pi \sum_{j=1}^N \delta_{p_j} \quad \text { in } \quad \Omega,
\end{equation}
where $\lambda$ is  a positive constant and $\delta_p$ is the Dirac distribution centred at $p \in \Omega$.

Recent investigations have expanded the exploration of the Chern-Simons model to encompass finite graphs. Notably, Huang, Lin, and Yau \cite{huang2020existence} conducted an in-depth study of equation (\ref{cye}), successfully establishing the existence of solutions for all but the critical case. Later,  Hou and Sun \cite{hou2022existence} addressed the critical case, resolving the existence of solutions therein. Furthermore, they extended their research to encompass a generalized form of the Chern-Simons equation. More studies on  the Chern-Simons model on graphs include \cite{huang2021mean,chao2022existence,CHAO2023126787,gao2022existence,hua2023existence} and so on. 

Variational methods have established themselves as an invaluable and robust tool for tackling partial differential equations on graphs. This methodology has been particularly effective in the study of complex equations like the Kazdan-Warner equations, as demonstrated in the works of Grigor et al. \cite{grigor2016kazdan} and others \cite{grigor2016kazdan,MR3648273,MR3776360}. The utility of variational methods extends to the domain of Schrödinger-type equations, which have been explored in  studies \cite{MR3833747,MR4092834}. Additionally, these methods have provided substantial insights into Yamabe-type equations, contributing  to our understanding of these complex problems \cite{MR3542963,MR3759076}. Equally noteworthy is their application in the realm of $p$-Laplacian equations, where they have facilitated advancements in both theory and application, as evidenced by recent research \cite{MR4344559,huaxu2023existence}.

Furthermore, the role of topological degree in the examination of nonlinear equations on graphs has been increasingly recognized as crucial. This concept has proven to be a key factor in unraveling the complexities of such equations, offering new perspectives and methodologies for their study, as highlighted in the latest literature \cite{MR4416135,MR4490503,li2023topological}. The integration of variational methods with topological degree concepts continues to be a dynamic and influential approach in the ongoing exploration and resolution of intricate equations on graph structures.

Let us denote \( V \) as the vertex set and \( E \) as the edge set. A finite graph can thus be represented as \( G = (V, E) \). We assume that \( G \) is connected, indicating that any pair of vertices can be connected via a finite sequence of edges. The weight on an edge \( xy \in E \) is defined as \( \omega_{xy} \), and is presumed to be symmetric, i.e., \( \omega_{xy} = \omega_{yx} \).

Consider \( \mu: V \rightarrow \mathbb{R}^{+} \) as a finite measure. The \( \mu \)-Laplacian operator for any function \( u: V \rightarrow \mathbb{R} \) is defined by
$$
\Delta u(x) = \frac{1}{\mu(x)} \sum_{y \sim x} \omega_{xy}(u(y) - u(x)),
$$
where \( y \sim x \) implies that \( xy \in E \). For any two functions \( u \) and \( \upsilon \), the gradient form is defined by
$$
\Gamma(u, \upsilon) = \frac{1}{2\mu(x)} \sum_{y \sim x} \omega_{xy}(u(y) - u(x))(\upsilon(y) - \upsilon(x)).
$$
When \( u = \upsilon \), we simply denote this as \( \Gamma(u) = \Gamma(u, u) \). Furthermore, define
$$
\| \nabla u \| (x) = \sqrt{\Gamma(u
	(x))} = \left( \frac{1}{2\mu(x)} \sum_{y \sim x} \omega_{xy}(u(y) - u(x))^2 \right)^{\frac{1}{2}}.$$
The integral over \( V \) is defined as
$$
\int_V u d\mu = \sum_{x \in V} \mu(x) u(x).
$$

In this paper, we consider   the following  Chern-Simons model on a graph \( G = (V, E) \): 
\begin{equation}\label{meq}
\Delta u=\lambda e^u\left(e^u-1\right)^{2p+1}+f,
\end{equation}
where $\Delta$ is the graph Laplacian, $\lambda$ is a real number, $p$ is a non-negative integer  and $f$ is a function on $V$.  If $p=0$ and $f = 4\pi \sum_{j=1}^{N} \delta_{p_j}$, where $p_1, \ldots, p_N \in V$, then equation (\ref{meq}) simplifies to equation (\ref{cye}). Conversely, if $p=2$ and $f = 4\pi \sum_{j=1}^{N} \delta_{p_j}$, equation (\ref{meq}) is reduced to equation (\ref{mef}).

Our objective is to use  the topological degree as demonstrated in \cite{MR4416135,li2023topological} for an in-depth analysis of the Chern-Simons Higgs model. The initial and foremost task in this endeavor is to acquire  a priori estimate of the solutions.

\begin{theorem}\label{th1} Let $ (V, E) $ represent a connected finite graph with symmetric weights, i.e., $ w_{xy}=w_{yx} $ for all $ xy \in E $. Suppose $ \sigma \in [0,1] $, $ \lambda $,  and $ f $ satisfy
\begin{equation}\label{thb}
\Lambda^{-1} \leq|\lambda| \leq \Lambda, \quad \Lambda^{-1} \leq\left|\int_V f d \mu\right| \leq \Lambda, \quad\|f\|_{L^{\infty}(V)} \leq \Lambda
\end{equation}
for some real number $ \Lambda>0 $. If $ u $ is a solution of

\begin{equation}\label{cs2}
\Delta u=\lambda e^u\left(e^u-\sigma\right)^{2p+1}+f \quad \text { in } \quad V,
\end{equation}
then there exists a constant $ C $, depending only on $ \Lambda $ and the graph $ V $, such that $ |u(x)| \leq C $ for all $ x \in V $.

\end{theorem}

Denote by $L=L^{\infty}(V)$ and define a map $\mathcal{F}=: L \rightarrow L$ by
\begin{equation}\label{cfd}
\mathcal{F}(u)=-\Delta u+\lambda e^u\left(e^u-1\right)^{2p+1}+f .
\end{equation}  Using Theorem \ref{th1}, we can calculate the topological degree of $\mathcal{F}$, utilizing its property of homotopic invariance. 
\begin{theorem}\label{th2} Let $(V, E)$ denote a connected, finite graph with symmetric weights. Consider a map \( \mathcal{F}: L \rightarrow L \) as defined in equation (\ref{cfd}). Assume \( \lambda \int_V f d \mu \neq 0 \). Under this assumption, the following deduction  holds: There exists a sufficiently large constant \( R_0 > 0 \), such that for all \( R \geq R_0 \), the degree of \( \mathcal{F} \) within the ball \( B_R \) at the origin is determined as follows:
$$
\operatorname{deg}\left(\mathcal{F}, B_R, 0\right)=\left\{\begin{array}{lll}
	1 & \text { if } & \lambda>0, \int_V f d \mu<0, \\
	0 & \text { if } & \lambda \int_V f d \mu>0 ,\\
	-1 & \text { if } & \lambda<0, \int_V f d \mu>0.
\end{array}\right.
$$
Here, \( B_R = \left\{ u \in L : \|u\|_{L^{\infty}(V)} < R \right\} \) is defined as a ball in the space \( L \).
\end{theorem}

Applying the aforementioned topological degree, our findings regarding the existence within the Chern-Simons Higgs model can be summarized as follows:
\begin{theorem}\label{th3}

Let $G=(V, E)$ be a connected finite graph  endowed with symmetric weights. The following results are established:

(a) If \( \lambda \int_V f \, d \mu < 0 \), then equation (\ref{meq}) admits a solution.

(b) If \( \lambda \int_V f \, d \mu > 0 \), two distinct subcases emerge:

(i) For \( \int_V f \, d \mu > 0 \), there exists a real number \( \Lambda^* > 0 \) such that equation (\ref{meq}) yields at least two distinct solutions for \( \lambda > \Lambda^* \), no solution for \( 0 < \lambda < \Lambda^* \), and at least one solution for \( \lambda = \Lambda^* \).

(ii) For \( \int_V f \, d \mu < 0 \), a real number \( \Lambda_* < 0 \) exists where equation (\ref{meq}) possesses at least two distinct solutions for \( \lambda < \Lambda_* \), no solution for \( \Lambda_* < \lambda < 0 \), and at least one solution for \( \lambda = \Lambda_* \).
\end{theorem}

This paper is organized as follows: Section 2 introduces  a priori estimate for solutions of equation \eqref{cs2}, elaborated in Theorem \ref{th1}. Section 3 involves the calculation of the topological degree of the mapping $\mathcal{F}\colon L \to L$, as established in Theorem \ref{th2}. Section 4 focuses on proving the existence result presented in Theorem \ref{th3}. We  draw inspiration from \cite{ huang2020existence,MR4416135, CHAO2023126787,li2023topological}.

\section{Proof of Theorem \ref{th1}}

In this section, we present the proof of Theorem \ref{th1}. The proof is based on priori estimates. 
	
Suppose that  $u$ is a solution of (\ref{cs2}).  By  integration of both sides of (\ref{cs2}) on $V$, we get 
\begin{equation} \label{sto}
	0=\int_V \Delta u d \mu=\lambda \int_V e^u\left(e^u-\sigma\right)^{2p+1} d \mu+\int_V f d \mu .
\end{equation}
Now we claim  that $u$ admits a uniform upper bound. Obviously, if $\max _V u<0$, then $0$ is an upper  bound. Hence, we assume that  $\max _V u>0$. There holds that 
	$$
	\left|\int_{u<0} e^u\left(e^u-\sigma\right)^{2p+1} d \mu\right| \leq |V|,
	$$
which  together with (\ref{sto}) yields that 
$$
	\int_{u \geq 0} e^u\left(e^u-\sigma\right)^{2p+1} d \mu \leq |V|+\frac{1}{|\lambda|}\left|\int_V f d \mu\right|\leq c_1:=|V|+\Lambda^2 .
$$
Observing that 
	$$
	\int_{u \geq 0} e^u\left(e^u-\sigma\right)^{2p+1} d \mu=\sum_{x \in V, u(x) \geq 0} \mu(x) e^{u(x)}\left(e^{u(x)}-\sigma\right)^{2p+1} \geq \mu_0 \left(e^{\max _V u}-\sigma\right)^{2p+1}
	$$
where $\mu_0=\min _{x \in V} \mu(x)>0$,  we conclude that 
\begin{equation}\label{mbo}\max _V u \leq \ln \left( 1+\left( \f{c_1}{\mu_0}\right)^{\f{1}{2p+1}}\right)\end{equation}
and confirm  the claim.

Next we prove that $u$ has a uniform lower bound. By (\ref{cs2}) and (\ref{mbo}), we obtain 
\begin{equation}
\begin{aligned}\label{lpa}|\Delta u(x)| & \leq|\lambda|\left|e^{u(x)}(e^{u(x)}-\sigma)^{2p+1}\right|+|f(x)| \\ & \leq |\lambda|\left|(e^{u(x)}+1)^{2p+2}\right|+|f(x)| \\ & \leq|\lambda|\left( 2+\left( \f{c_1}{\mu_0}\right)^{\f{1}{2p+1}}\right)^{2p+2}+\|f\|_{L^{\infty}(V)} \\ &
\leq \Lambda\left( 2+\left( \f{c_1}{\mu_0}\right)^{\f{1}{2p+1}}\right)^{2p+2}+\Lambda\\
 &  =: c_2 .
\end{aligned}
\end{equation}
Noting that $G$ is finite, let  $V=\left\{x_1, \cdots, x_{m}\right\}.$ Assume that $ u\left(x_1\right)=\min _V u, u\left(x_{m}\right)=\max _V u$, and without loss of generality $x_1 x_2, x_2 x_3, \cdots, x_{m-1} x_{m}$ is the shortest path connecting $x_1$ and $x_{m}$. Denote by  $w_0=\min _{x \in V, y \sim x} w_{x y}>0$. It is easy to see that 
\begin{equation}\label{lpe}
\begin{aligned}
	0 \leq u\left(x_m\right)-u\left(x_{1}\right) & \leq \sum_{j=1}^{m-1}\left|u(x_{j+1})-u(x_{j})\right| \\
	& \leq \frac{\sqrt{m-1}}{\sqrt{w_0}}\left(\sum_{j=1}^{m-1} w_{x_j x_{j+1}}\left(u(x_{j+1})-u(x_{j})\right)^2\right)^{1 / 2} \\
	& \leq \frac{\sqrt{m-1}}{\sqrt{w_0}}\left(\int_V|\nabla u|^2 d \mu\right)^{1 / 2}.
\end{aligned}
\end{equation}
Denote by  $\bar{u}=\frac{1}{|V|} \int_V u d \mu$ and  by $\lambda_1=\inf _{\bar{u}=0, \int_V u^2 d \mu=1} \int_V|\nabla u|^2 d \mu>0$.  Integration by parts on $V$ gives 
$$
\begin{aligned}
	\int_V|\nabla u|^2 d \mu & =-\int_V(u-\bar{u}) \Delta u d \mu \\
	& \leq\left(\int_V(u-\bar{u})^2 d \mu\right)^{1 / 2}\left(\int_V(\Delta u)^2 d \mu\right)^{1 / 2} \\
	& \leq\left(\frac{1}{\lambda_1} \int_V|\nabla u|^2 d \mu\right)^{1 / 2}\left(\int_V(\Delta u)^2 d \mu\right)^{1 / 2},
\end{aligned}
$$
which implies 
\begin{equation}\label{gre}\int_V|\nabla u|^2 d \mu \leq \frac{1}{\lambda_1} \int_V(\Delta u)^2 d \mu \leq \frac{1}{\lambda_1}\|\Delta u\|_{L^{\infty}(V)}^2|V|.
\end{equation}
In view of (\ref{lpe}) and (\ref{gre}), we arrive at 
\begin{equation}\label{mmb}
\max _V u-\min _V u \leq \sqrt{\frac{(m-1)|V|}{w_0 \lambda_1}}\|\Delta u\|_{L^{\infty}(V)} .
\end{equation}
Furthermore, by (\ref{lpa}), we have 
\begin{equation}\label{mme}
\max _V u-\min _V u \leq c_3:=c_2\sqrt{\frac{(m-1)|V|}{w_0 \lambda_1}}.
\end{equation}
In view of (\ref{sto}), we obtain 
\begin{equation}\label{les}
\left |\int_V e^u\left(e^u-\sigma\right)^{2p+1} d \mu\right|=\left|-\frac{1}{\lambda} \int_V f d \mu\right|\geq \frac{1}{\Lambda^2}.
\end{equation}
We claim that
 \begin{equation}\label{lwe}
\max _V u>-\alpha:=\ln\min \left\{1, \frac{1}{4\Lambda^2|V|}\right\}
. 
\end{equation} 
We prove it by contradiction. Suppose   $
\max _V u \leq-\alpha.
$ 
Noting that if $e^u\geq 0$,  then $\left| e^u(e^u-\sigma)^{2p+1}\right|\leq e^{(2p+2)u}$, while $\left| e^u(e^u-\sigma)^{2p+1}\right|\leq e^{u}$ if $e^u\leq 0$, hence there always holds that 
\begin{equation}\label{eel}
	\left| e^u(e^u-\sigma)^{2p+1}\right|\leq \left( e^{(2p+2)u}+e^u\right).
\end{equation}
By (\ref{les}) and (\ref{eel}), we obtain that 
$$
\begin{aligned}
	\frac{1}{\Lambda^2} & \leq \left|\int_V e^u\left(e^u-\sigma\right)^{2p+1} d \mu\right| \\
	& \leq \int_V\left(e^{(2p+2)u}+e^u\right) d \mu \\
	& \leq\left(e^{(2p+2) \max _V u}+e^{\max _V u}\right)|V| \\
	& \leq 2 e^{-\alpha}|V| \\
	& <\frac{1}{2\Lambda^2},
\end{aligned}
$$
which yields a contradiction.  Hence the claim holds. Combing (\ref{mbo}), (\ref{mme}) and (\ref{lwe}), we get $$-\alpha-c_3 \leq \min _V u \leq \max _V u \leq \ln \left( 1+\left( \f{c_1}{\mu_0}\right)^{\f{1}{2p+1}}\right).$$  We conclude that $|u(x)| \leq C$, where $C$ depends only on $ \Lambda $ and  $ V $.

\section{Proof of Theorem \ref{th2}}
In this section, we will prove Theorem \ref{th2} by the topological degree.  Consider the set \( V = \left\{ x_1, \cdots, x_{m} \right\} \). Let us define \( L \) as \( L^{\infty}(V) \). In this context, \( L \) can be identified with the Euclidean space \( \mathbb{R}^{m} \). To avoid any ambiguity, we introduce a mapping \( \mathcal{F}: L \times [0,1] \rightarrow L \) defined by the following equation:
$$
\mathcal{F}(u, \sigma)=-\Delta u+\lambda e^u\left(e^u-\sigma\right)^{2p+1}+f, \quad(u, \sigma) \in L \times[0,1] .
$$
In this formulation, \( \mathcal{F} \) encapsulates the interaction within the space \( L \) and the interval $[0,1]$.

Given a fixed real number \( \lambda \) and a fixed function \( f \), and considering \( \lambda \bar{f} \neq 0 \), there necessarily exists a substantial number \( \Lambda > 0 \) satisfying the following conditions:
\begin{equation}\label{slu}
\Lambda^{-1} \leq |\lambda| \leq \Lambda, \quad \Lambda^{-1} \leq \left|\int_V f d \mu\right| \leq \Lambda, \quad \|f\|_{L^{\infty}(V)} \leq \Lambda.
\end{equation}
Here, and in subsequent discussions, \( \bar{f} \) represents the integral mean of the function \( f \). Based on Theorem \ref{th1}, it can be inferred that there is a constant \( \bar{R} > 0 \), and the graph \( V \), such that for every \( \sigma \in [0,1] \), all solutions to the equation \( \mathcal{F}(u, \sigma) = 0 \) will comply with the condition \( \|u\|_{L^{\infty}(V)} < \bar{R} \). Let \( B_r \subset L \) denote a ball centered at the origin in \( L \) with radius \( r \), and \( \partial B_r = \left\{ u \in L : \|u\|_{L^{\infty}(V)} = r \right\} \) denote its boundary. Consequently, 
for all \( \sigma \in [0,1] \) and \( R \geq \bar{R} \), it is observed that:
$$
0 \notin \mathcal{F}\left(\partial B_R, \sigma\right).
$$
Invoking the principle of homotopic invariance of topological degree, we deduce:
\begin{equation}\label{bde}
\operatorname{deg}\left(\mathcal{F}(\cdot, 1), B_R, 0\right) = \operatorname{deg}\left(\mathcal{F}(\cdot, 0), B_R, 0\right), \quad \forall R \geq \bar{R} .
\end{equation}
For a given \( \epsilon > 0 \), we introduce an alternative smooth map \( \mathcal{G}_\epsilon: L \times [0,1] \rightarrow L \) defined as:

$$
\mathcal{G}_\epsilon(u, t) = -\Delta u + \lambda e^{(2p+2)u} + (t + (1 - t) \epsilon) f, \quad (u, t) \in L\times [0,1] .
$$
It is noteworthy that:

$$
\min \{1, \epsilon\} \left| \int_V f d \mu \right|\leq\left| (t + (1 - t) \epsilon) \int_V f d \mu \right| \leq \left| \int_V f d \mu \right| , \quad \forall t \in [0,1] .
$$
Applying Theorem \ref{th1} once more, we identify a constant \( R_\epsilon > 0 \), dependent exclusively on \( \epsilon \), \( \Lambda \), and the graph \( V \). This ensures that all solutions \( u \) to \( \mathcal{G}_\epsilon(u, t) = 0 \) satisfy \( \|u\|_{L^{\infty}(V)} < R_\epsilon \) for all \( t \in [0,1] \). Consequently, this leads to the implication:

$$
0 \notin \mathcal{G}_\epsilon\left(\partial B_{R_{\epsilon}}, t\right), \quad \forall t \in [0,1] .
$$
Therefore, the principle of homotopic invariance in topological degree yields:
\begin{equation}\label{gde}
\operatorname{deg}\left(\mathcal{G}_\epsilon(\cdot, 1), B_{R_\epsilon}, 0\right) = \operatorname{deg}\left(\mathcal{G}_\epsilon(\cdot, 0), B_{R_\epsilon}, 0\right).
\end{equation}
To compute \( \operatorname{deg}\left(\mathcal{G}_\epsilon(\cdot, 0), B_{R_\epsilon}, 0\right) \), it is essential to examine the solvability of the equation:
\begin{equation}\label{mge}
\mathcal{G}_\epsilon(u, 0) = -\Delta u + \lambda e^{(2p+2)u} + \epsilon f = 0.
\end{equation}

We consider two cases: (i) $\lambda \bar{f}<0$; (ii) $\lambda \bar{f}>0$.

Next we deal the case $\lambda \bar{f}<0$.  For any given \( \epsilon > 0 \), let \( v_\epsilon \) denote the unique solution to the equation:
$$
\left\{\begin{array}{l}
	\Delta v = \epsilon f - \epsilon \bar{f} \text{ in } V \\
	\bar{v} = 0 .
\end{array}\right.
$$
Consequently, letting $w=u-v_\epsilon$, we reduce  (\ref{mge})  to the following equation 
$$
\Delta w = \lambda e^{2(p+1) v_\epsilon} e^{2(p+1) w} + \epsilon \bar{f} .
$$
Set $W=2(p+1)w$. Then we get 
\begin{equation}\label{nwe}
\Delta W = 2\lambda(p+1) e^{2(p+1) v_\epsilon} e^{W} + 2(p+1)\epsilon \bar{f} .
\end{equation}
By Theorems 2 and 4 in \cite{grigor2016kazdan}, we know that there $\epsilon_1>0$ such that (\ref{nwe}) admits a solution  if $0<\epsilon<\epsilon_1$ as $\lambda \bar{f}<0$. Hence, (\ref{mge}) has a solution. 

Noting (\ref{slu}) and integrating both sides of (\ref{mge}) on $V$, we get 
$$
\int_V e^{2(p+1) u_\epsilon} d \mu = -\frac{\epsilon}{\lambda} \int_V f d \mu \leq \Lambda^2 \epsilon,
$$
which implies that
\begin{equation}\label{u2p}
e^{2(p+1) u_\epsilon(x)} \leq \frac{\Lambda^2}{\mu_0} \epsilon, \quad \forall x \in V,
\end{equation}
where \( \mu_0 \) is defined as the minimum value of \( \mu(x) \) over all \( x \in V \). Additionally, it is imperative to establish the uniqueness of the solution. Assume \( \varphi \) represents an arbitrary solution of equation (\ref{mge}), thereby satisfying
\begin{equation}\label{ans}
\Delta \varphi = \lambda e^{2(p+1) \varphi} + \epsilon f .
\end{equation}
By the arguments as before, we have 
\begin{equation}\label{phe}
	\int_V e^{2(p+1) \varphi} d \mu \leq \Lambda^2 \epsilon, \quad e^{2(p+1) \varphi(x)} \leq \frac{\Lambda^2}{\mu_0} \epsilon  \quad\text { for all } \quad x \in V. 
\end{equation} 
It follows from (\ref{mge}) and (\ref{ans}) that 
$$
0 = \int_V \Delta(u_\epsilon - \varphi) d \mu = \lambda \int_V (e^{2(p+1) u_\epsilon} - e^{2(p+1) \varphi}) d \mu,
$$
which subsequently leads to
$$
\min_V (u_\epsilon - \varphi) \leq 0 \leq \max_V (u_\epsilon - \varphi) .
$$
As a direct consequence, we can infer that
\begin{equation}\label{uva}
|u_\epsilon - \varphi| \leq \max_V (u_\epsilon - \varphi) - \min_V (u_\epsilon - \varphi) .
\end{equation}
Further, by synthesizing equations (\ref{mge}), (\ref{u2p}), (\ref{ans}), and (\ref{phe}), we deduce
\begin{equation}\label{dai}
\begin{aligned}
	|\Delta(u_\epsilon - \varphi)(x)| &= \left|\lambda \left(e^{2(p+1) u_\epsilon(x)} - e^{2(p+1) \varphi(x)}\right)\right| \\
	&\leq 2(p+1) \Lambda \left(e^{2(p+1) u_\epsilon(x)} + e^{2(p+1) \varphi(x)}\right) |u_\epsilon(x) - \varphi(x)| \\
	&\leq \frac{4(p+1) \Lambda^3}{\mu_0} \epsilon |u_\epsilon(x) - \varphi(x)| .
\end{aligned}
\end{equation}
Integrating the results from equations (\ref{mmb}), (\ref{uva}), and (\ref{dai}), we arrive at the following inequality:
\begin{equation}\label{mmv}
\max_V (u_\epsilon - \varphi) - \min_V (u_\epsilon - \varphi) \leq \sqrt{\frac{(m- 1) |V|}{w_0 \lambda_1}} \frac{4(p+1) \Lambda^3}{\mu_0} \epsilon (\max_V (u_\epsilon - \varphi) - \min_V (u_\epsilon - \varphi)) .
\end{equation}
We then select
$$
\epsilon_0 = \min \left\{\epsilon_1, \sqrt{\frac{w_0 \lambda_1}{(m- 1) |V|}} \frac{\mu_0}{8(p+1) \Lambda^3}\right\} .
$$
By choosing \( 0 < \epsilon < \epsilon_0 \), equation (\ref{mmv}) leads to the conclusion that \( \varphi \) is identically equal to \( u_\epsilon \) on \( V \). Consequently, this implies the uniqueness of the solution to equation (\ref{mge}). 

It is important to note that \( -\Delta: L \rightarrow L \) is a nonnegative definite symmetric operator. Its eigenvalues can be represented as:
$$
0 = \lambda_0 < \lambda_1 \leq \lambda_2 \leq \cdots \leq \lambda_{m - 1},
$$
where \( m \) denotes the total number of points in \( V \). Noting (\ref{u2p}), choose a sufficiently small $\epsilon $  such that the  unique solution \( u_\epsilon \) to (\ref{mge}) satisfying
$$
2(p+1)|\lambda| e^{2(p+1) u_\epsilon(x)} < \lambda_1 .
$$
A direct computation reveals that
$$
D \mathcal{G}_\epsilon(u_\epsilon, 0) = -\Delta + 2(p+1) \lambda e^{2(p+1) u_\epsilon} \mathrm{I}
$$
wherein the linear operator \( -\Delta \) is identified with the corresponding \( m\times m \) matrix and \( \mathrm{I} \) denotes the \( m \times m \) diagonal matrix \( \operatorname{diag}[1,1, \cdots, 1] \). Evidently,
\begin{equation}\begin{aligned}
\operatorname{deg}(\mathcal{G}_\epsilon(\cdot, 0), B_{R_\epsilon}, 0) =& \operatorname{sgn} \operatorname{det}(D \mathcal{G}_\epsilon(u_\epsilon, 0))\\
 = &\operatorname{sgn}\left\{2(p+1) \lambda e^{2(p+1) u_\epsilon(x)} \prod_{j=1}^{m-1}(\lambda_j + 2(p+1) \lambda e^{2(p+1) u_\epsilon(x)})\right\}\\
  = &\operatorname{sgn} \lambda.
\end{aligned}\end{equation}
This, in conjunction with equations (\ref{bde}) and (\ref{gde}), culminates in the following chain of equalities:
$$
\begin{aligned}
	\operatorname{deg}(\mathcal{F}(\cdot, 1), B_{R}, 0) & = \operatorname{deg}(\mathcal{F}(\cdot, 0), B_{R}, 0) \\
	& = \operatorname{deg}(\mathcal{G}_\epsilon(\cdot, 1), B_{R}, 0) \\
	& = \operatorname{deg}(\mathcal{G}_\epsilon(\cdot, 0), B_{R}, 0) \\
	& = \operatorname{sgn} \lambda,
\end{aligned}
$$
where $R\geq R_0:=\max\{ \bar{R},R_{\epsilon}\}$. Hence if  $\lambda>0, \int_V f d \mu<0$, $\operatorname{deg}\left(\mathcal{F}, B_R, 0\right)=1$ and if  $\lambda<0, \int_V f d \mu>0$, $\operatorname{deg}\left(\mathcal{F}, B_R, 0\right)=-1$. 

Now we consider the case  $\lambda \bar{f}>0$.  If (\ref{mge}) admits a solution $u_{\epsilon}$, it follows that 
$$
0=\int_V \Delta u_{\epsilon} d \mu=\lambda \int_V e^{2(p+1) u_{\epsilon}} d \mu+\epsilon \int_V f d \mu,
$$
which is impossible. Hence 
$$
\operatorname{deg}\left(\mathcal{F}(\cdot, 1), B_{R}, 0\right)=\operatorname{deg}\left(G_\epsilon(\cdot, 0), B_{R}, 0\right)=0 .
$$
We finish the proof of Theorem \ref{th2}. 

\section{Proof of Theorem \ref{th3}}

In this section, our aim is to establish the validity of Theorem \ref{th3} by utilizing the topological degree as outlined in Theorem \ref{th2}.

 Assume \( \lambda \bar{f} < 0 \). In light of Theorem \ref{th2}, we can identify a sufficiently large \( R_0 > 1 \) such that
$$
\operatorname{deg}(\mathcal{F}, B_{R_0}, 0) \neq 0 .
$$
Consequently, according to Kronecker's existence theorem, equation (\ref{meq}) is guaranteed to have a solution. Hence we prove (a) in Theorem \ref{th3}. 

In the subsequent discussions of this section, we shall proceed under the assumption that \( \lambda \bar{f} > 0 \). Our initial focus will be to demonstrate that equation (\ref{meq}) admits a local minimum solution when \( |\lambda| \) is large enough.

\begin{lemma}\label{l41}
For a sufficiently large  $|\lambda |$,  equation (\ref{meq}) is guaranteed to have a local minimum solution.
\end{lemma} 
\begin{proof}
We will analyze two distinct cases. The first case pertains to the conditions where \( \lambda > 0 \) and \( \bar{f} > 0 \). The second case involves situations where \( \lambda < 0 \) and \( \bar{f} < 0 \).
	
Initially, let's focus on the case where \( \lambda > 0 \) and \( \bar{f} > 0 \). We define the operator \( \mathcal{L}_\lambda \) as follows:
\begin{equation}\label{dfl}
\mathcal{L}_\lambda u = -\Delta u + \lambda e^u(e^u - 1)^{2p+1} + f .
\end{equation}
For the real numbers \( A \) and \( \lambda \), the following relationships are observed:
$$
\mathcal{L}_\lambda A = \lambda e^A(e^A - 1)^{2p+1} + f, \quad \mathcal{L}_\lambda \ln \frac{1}{2} = -\frac{1}{2^{2p+2}} \lambda + f .
$$
It becomes evident that by choosing sufficiently large values of \( A > 1 \) and \( \lambda > 1 \), we can ensure
\begin{equation}\label{lul}
\mathcal{L}_\lambda A > 0, \quad \mathcal{L}_\lambda \ln\frac{1}{2} < 0 .
\end{equation}
Define  the functional \( \mathcal{J}_\lambda: L= L^{\infty}(V) \rightarrow \mathbb{R} \) by 
\begin{equation}\mathcal{J}_\lambda(u)=\frac{1}{2} \int_V|\nabla u|^2 d \mu+\frac{\lambda}{2(p+1)} \int_V\left(e^u-1\right)^{2(p+1)} d \mu+\int_V f u d \mu.
\end{equation}
 Noting that \( L \cong \mathbb{R}^{m} \) and \( \mathcal{J}_\lambda \in C^2(L, \mathbb{R}) \), we consider the bounded closed subset \( \{u \in L : \ln\frac{1}{2} \leq u \leq A\} \) of \( L \). It is straightforward to identify some \( u_\lambda \in L \) satisfying \( \ln\frac{1}{2} \leq u_\lambda(x) \leq A \) for all \( x \in V \) and
\begin{equation}\label{dej}
\mathcal{J}_\lambda(u_\lambda) = \min_{\ln \frac{1}{2} \leq u \leq A} \mathcal{J}_\lambda(u) .
\end{equation}
We assert that
\begin{equation}\label{asi}
\ln \frac{1}{2} < u_\lambda(x) < A \quad\text{ for all } \quad x \in V.
\end{equation}
Let us assume the contrary. It must be the case that \( u_\lambda(x_0) = \ln \frac{1}{2} \) for some \( x_0 \in V \), or \( u_\lambda(x_1) = A \) for some \( x_1 \in V \). If \( u_\lambda(x_0) = \ln \frac{1}{2} \), we choose a small \( \epsilon > 0 \) such that
$$
\ln\frac{1}{2} \leq u_\lambda(x) + t \delta_{x_0}(x) \leq A, \quad \forall x \in V, \,\,\forall t \in (0, \epsilon) .
$$
On one hand, considering equations (\ref{lul}) and (\ref{dej}), we derive
\begin{equation}\label{jlu}
\begin{aligned}
	0 &\leq \left.\frac{d}{d t}\right|_{t=0} \mathcal{J}_\lambda(u_\lambda + t \delta_{x_0}) \\
	&= \int_V (-\Delta u_\lambda + \lambda e^{u_\lambda}(e^{u_\lambda} - 1)^{2p+1} + f) \delta_{x_0} \, d \mu \\
	&= -\Delta u_\lambda(x_0) + \lambda e^{u_\lambda(x_0)}(e^{u_\lambda(x_0)} - 1)^{2p+1} + f(x_0) \\
	&< -\Delta u_\lambda(x_0) .
\end{aligned}
\end{equation}
On the other hand, since \( u_\lambda(x) \geq u_\lambda(x_0) \) for all \( x \in V \), it follows that \( \Delta u_\lambda(x_0) \geq 0 \), leading to a contradiction with equation (\ref{jlu}). Therefore, it must hold that \( u_\lambda(x) > \ln\frac{1}{2} \) for all \( x \in V \). Similarly, we can exclude the possibility that \( u_\lambda(x_1) = A \) for some \( x_1 \in V \). This verification substantiates our assertion  (\ref{asi}). By synthesizing equations (\ref{dej}) and (\ref{asi}), we conclude that \( u_\lambda \) is a local minimum critical point of \( \mathcal{J}_\lambda \), and specifically, a solution of equation (\ref{meq}).

We now turn our attention to the subcase where \( \lambda < 0 \) and \( \bar{f} < 0 \). Consider \( \varphi \), the unique solution to the system
\begin{equation}\label{vse}
\left\{\begin{array}{l}
	\Delta \varphi = f - \bar{f} \\
	\bar{\varphi} = 0 .
\end{array}\right.
\end{equation}
Employing the previously defined operator \( \mathcal{L}_\lambda \) from equation (\ref{dfl}), we can deduce that
\begin{equation}\label{cli}
\begin{aligned}
	\mathcal{L}_\lambda(\varphi - A) &= -\Delta \varphi + \lambda e^{\varphi - A}(e^{\varphi - A} - 1)^{2p+1} + f \\
	&= \lambda e^{\varphi - A}(e^{\varphi - A} - 1)^{2p+1} + \bar{f} \\
	&< 0
\end{aligned}
\end{equation}
and
$$
\begin{aligned}
	\mathcal{L}_\lambda\left(\ln \frac{1}{2}\right) &= \lambda e^{\ln \frac{1}{2}}\left(e^{\ln\frac{1}{2}} - 1\right)^{2p+1} + f \\
	&= -\frac{\lambda}{2^{2p+2}} + f \\
	&> 0,
\end{aligned}
$$
assuming \( \lambda < 2^{2p+2} \min_V f \) and \( A > 1 \) is chosen to be sufficiently large. Following a similar approach to equations (\ref{dej}) and (\ref{asi}), there exists a certain \( u_\lambda \) satisfying \( \varphi(x) - A < u_\lambda(x) < \ln\frac{1}{2} \) for all \( x \in V \), and
$$
\mathcal{J}_\lambda(u_\lambda) = \min_{\varphi - A \leq u \leq \ln \frac{1}{2}} \mathcal{J}_\lambda(u) = \min_{\varphi - A < u < \ln \frac{1}{2}} \mathcal{J}_\lambda(u)
$$
This leads us to conclude that \( u_\lambda \) is a local minimum solution of equation (\ref{meq}).
\end{proof}

Next, we demonstrate that the range of $\l$ for which equation (\ref{meq}) admits a local minimum solution encompasses two intervals.
\begin{lemma}\label{l42}
Assume that \( L_{\lambda_1} u_{\lambda_1} = L_{\lambda_2} u_{\lambda_2} = 0 \) holds true on the set \( V \). If the conditions \( \lambda > \lambda_1 > 0 \) or \( \lambda < \lambda_2 < 0 \) are satisfied, it follows that equation (\ref{meq}) admits a local minimum solution, denoted as \( u_\lambda \).
	
\end{lemma}

\begin{proof}
Let us assume that \( \lambda > \lambda_1 > 0 \). Choose \( A > 1 \) to be a sufficiently large constant such that \(\mathcal{L}_\lambda A > 0 \) and \( u_{\lambda_1} + \ln\frac{\lambda_1}{\lambda} < A \) hold true on \( V \). Under these conditions, there exists a function \( u_\lambda \) satisfying
$$
\mathcal{J}_\lambda(u_\lambda) = \min_{u_{\lambda_1} + \ln \frac{\lambda_1}{\lambda} \leq u \leq A} \mathcal{J}_\lambda(u).
$$

Suppose there exists a point \( x_0 \in V \) where \( u_\lambda(x_0) = u_{\lambda_1}(x_0) + \ln\frac{\lambda_1}{\lambda} \). Select a small \( \epsilon > 0 \) such that for all \( t \in (0, \epsilon) \), the following inequality is satisfied:
$$
u_{\lambda_1}(x) + \ln \frac{\lambda_1}{\lambda} \leq  u_\lambda(x) + t \delta_{x_0}(x) \leq  A \quad \text{for all} \quad x \in V .
$$
In a manner analogous to the approach used in the proof of Lemma \ref{l41}, we derive
$$
\begin{aligned}
	0 &\leq \left.\frac{d}{dt}\right|_{t=0} \mathcal{J}_\lambda(u_\lambda + t \delta_{x_0}) \\
	&= -\Delta u_\lambda(x_0) + \lambda e^{u_\lambda(x_0)}(e^{u_\lambda(x_0)} - 1)^{2p+1} + f(x_0) \\
	&= -\Delta(u_\lambda - u_{\lambda_1})(x_0)-\Delta u_{\lambda_1}(x_0)+\lambda_1 e^{u_{\lambda_1}(x_0)}\left(\frac{\lambda_1}{\lambda} e^{u_{\lambda_1}(x_0)}-1\right)^{2p+1}+f(x_0) \\
	&<-\Delta(u_\lambda - u_{\lambda_1})(x_0)-\Delta u_{\lambda_1}(x_0)+\lambda_1 e^{u_{\lambda_1}(x_0)}\left( e^{u_{\lambda_1}(x_0)}-1\right)^{2p+1}+f(x_0)\\
	&=-\Delta(u_\lambda - u_{\lambda_1})(x_0).
\end{aligned}
$$
This result contradicts the established fact that \( x_0 \) is a minimum point of \( u_\lambda - u_{\lambda_1} - \ln \frac{\lambda_1}{\lambda} \). Consequently, it must hold that
$$
u_\lambda(x) > u_{\lambda_1}(x) + \ln \frac{\lambda_1}{\lambda}, \quad \forall x \in V .
$$
Similarly, we can conclude that \( u(x) < A \) for all \( x \in V \). Therefore, \( u_\lambda \) is a local minimum critical point of \( \mathcal{J}_\lambda \).

Assume \( \lambda < \lambda_2 < 0 \). We select a sufficiently large constant \( A > 1 \) such that \( \varphi - A < u_{\lambda_2} + \ln \frac{\lambda_2}{\lambda} \) on \( V \), and ensure that \( \varphi - A \) fulfills equation (\ref{cli}). It is evident that there exists some \( u_\lambda \) satisfying
$$
\mathcal{J}_\lambda(u_\lambda) = \min_{\varphi - A \leq u \leq u_{\lambda_2} + \ln \frac{\lambda_2}{\lambda}} \mathcal{J}_\lambda(u) .
$$
If a point \( x_1 \in V \) exists such that \( u_\lambda(x_1) = u_{\lambda_2}(x_1) + \ln\frac{\lambda_2}{\lambda} \), then for a small \( \epsilon > 0 \) and for all \( t \in (0, \epsilon) \), the following condition holds:
$$
\varphi(x) - A \leq u_\lambda(x) - t \delta_{x_1}(x) \leq u_{\lambda_2}(x) + \ln\frac{\lambda_2}{\lambda} \quad \text{for all} \quad x \in V .
$$
Utilizing the same method as previously discussed, we arrive at
$$
\begin{aligned}
	0 &\leq \left.\frac{d}{dt}\right|_{t=0} \mathcal{J}_\lambda(u_\lambda - t \delta_{x_1}) \\
	&= \Delta u_\lambda(x_1) - \lambda e^{u_\lambda(x_1)}(e^{u_\lambda(x_1)} - 1)^{2p+1} - f(x_1) \\
	&= \Delta(u_\lambda - u_{\lambda_2})(x_1) +\Delta u_{\lambda_2}(x_1)-\lambda_2 e^{u_{\lambda_2}(x_1)}\left(\frac{\lambda_2}{\lambda} e^{u_{\lambda_2}(x_1)}-1\right)^{2p+1}-f(x_1)\\
	&<\Delta(u_\lambda - u_{\lambda_2})(x_1) +\Delta u_{\lambda_2}(x_1)-\lambda_2 e^{u_{\lambda_2}(x_1)}\left( e^{u_{\lambda_2}(x_1)}-1\right)^{2p+1}-f(x_1)\\
	&=\Delta(u_\lambda - u_{\lambda_2})(x_1).
\end{aligned}
$$
This result is in contradiction with the established fact that \( x_1 \) is a maximum point of \( u_\lambda - u_{\lambda_2} - \ln \frac{\lambda_2}{\lambda} \). Therefore, we conclude that
$$
u_\lambda(x) < u_{\lambda_2}(x) + \ln \frac{\lambda_2}{\lambda}, \quad \forall x \in V .
$$
Similarly, we deduce that \( u(x) > \varphi(x) - A \) for all \( x \in V \). Hence, \( u_\lambda \) is a local minimum critical point of \( \mathcal{J}_\lambda \). This completes the proof of the lemma.
\end{proof}

Drawing conclusions from Lemmas \ref{l41} and \ref{l42}, we ascertain that the following two critical values are well-defined:
\begin{equation}\label{lud}
\Lambda^* = \inf \left\{ \lambda > 0 : \lambda \bar{f} > 0, \text{and } \mathcal{J}_\lambda \text{ possesses a local minimum critical point} \right\} 
,\end{equation}
\begin{equation}\label{lld}
\Lambda_* = \sup \left\{ \lambda < 0 : \lambda \bar{f} > 0, \text{and } \mathcal{J}_\lambda \text{ possesses a local minimum critical point} \right\} .
\end{equation}

Next, a lower bound for $\Lambda^*$ and an upper bound for $\Lambda_*$ are established.
\begin{lemma} If \( \bar{f} > 0 \), then \( \Lambda^* \geq \frac{(2p+2)^{2p+2}}{(2p+1)^{2p+1}} \bar{f} \); if \( \bar{f} < 0 \), then \( \Lambda_* \leq \frac{(2p+2)^{2p+2}}{(2p+1)^{2p+1}} \bar{f} \).
\end{lemma}
\begin{proof} Let us assume \( \lambda \neq 0 \) and that \( u \) is a solution of the equation \( \Delta u = \lambda e^u(e^u - 1)^{2p+1} + f \). Integration by parts yields
$$
-\frac{\int_V f \, d \mu}{\lambda} = \int_V e^u(e^u - 1)^{2p+1} \, d \mu \geq -\frac{(2p+1)^{2p+1}|V|}{(2p+2)^{2p+2}},
$$
since \( e^u(e^u - 1)^{2p+1} \geq -\frac{(2p+1)^{2p+1}}{(2p+2)^{2p+2}} \). The conclusion is then immediately derived from  (\ref{lud}) and (\ref{lld}).
\end{proof}
We are now well-positioned to finalize the proof of the remaining part of the theorem.

We first examine the solvability of equation (\ref{meq}) under the assumption that \( \lambda \in (0, \Lambda^*] \cup [\Lambda_*, 0) \).

If \( \lambda \in (0, \Lambda^*) \cup (\Lambda_*, 0) \), then equation (\ref{meq}) has no solution. Indeed, assume there is a \( \lambda_1 \in (0, \Lambda^*) \cup (\Lambda_*, 0) \) such that equation (\ref{meq}) is solvable at \( \lambda = \lambda_1 \). Without loss of generality, let \( \lambda_1 \in (0, \Lambda^*) \). Then, according to Lemma \ref{l42}, equation (\ref{meq}) possesses a local minimum solution for any \( \lambda \in (\lambda_1, \Lambda^*] \), contradicting the definition of \( \Lambda^* \). Therefore, equation (\ref{meq}) is unsolvable for any \( \lambda \in (0, \Lambda^*) \cup (\Lambda_*, 0) \).

Moreover, for any \( i\in \mathbb{N} \), a solution \( u_i \) of (\ref{meq}) exists with \( \lambda = \Lambda^* + 1/i \). Based on Theorem \ref{th1}, the sequence \( (u_i) \) is uniformly bounded in \( V \). Thus, a subsequence of \( (u_i) \) uniformly converges to some function \( u^* \), a solution of (\ref{meq}) with \( \lambda = \Lambda^* \). Similarly, equation (\ref{meq}) also admits a solution at \( \lambda = \Lambda_* \).

Next, we consider the existence of multiple solutions of (\ref{meq}) under the assumption \( \lambda \in (\Lambda^*, +\infty) \cup (-\infty, \Lambda_*) \). If \( \lambda \in (\Lambda^*, +\infty) \cup (-\infty, \Lambda_*) \), by equations (\ref{lud}) and (\ref{lld}), let \( u_\lambda \) be a local minimum critical point of \( \mathcal{J}_\lambda \). Assuming \( u_\lambda \) as the unique critical point of \( \mathcal{J}_\lambda \) (otherwise, \( J_\lambda \) already has at least two critical points and the proof concludes). According to [\cite{MR1196690}, Chapter 1, Page 32], the \( q \)-th critical group of \( \mathcal{J}_\lambda \) at \( u_\lambda \) is defined by
$$
\mathsf{C}_q(\mathcal{J}_\lambda, u_\lambda) = \mathsf{H}_q(\mathcal{J}_\lambda^c \cap U, \{\mathcal{J}_\lambda^c \backslash \{u_\lambda\}\} \cap U, \mathsf{G}),
$$
where \( \mathcal{J}_\lambda(u_\lambda) = c \), \( \mathcal{J}_\lambda^c = \{u \in L: \mathcal{J}_\lambda(u) \leq c\} \), \( U \) is a neighborhood of \( u_\lambda \) in \( L \), and \( \mathsf{H}_q \) is the singular homology group with coefficients in group \( G \) (e.g., \( \mathbb{Z}, \mathbb{R} \)). The excision property of \( \mathsf{H}_q \) ensures this definition is independent of \( U \)'s choice. It is straightforward to compute
\begin{equation}\label{cqg}
\mathsf{C}_q(\mathcal{J}_\lambda, u_\lambda) = \delta_{q0} \mathsf{G}.
\end{equation}
We show that \( \mathcal{J}_\lambda \) satisfies the Palais-Smale condition. If \( \mathcal{J}_\lambda(u_i) \rightarrow c \in \mathbb{R} \) and \( J_\lambda'(u_i) \rightarrow 0 \) as \( i \rightarrow \infty \), then using the method from Theorem \ref{th1}, the sequence \( (u_i) \) is uniformly bounded. Given that \( L \) is precompact, a subsequence of \( (u_i) \) converges uniformly to some \( u^* \), a critical point of \( \mathcal{J}_\lambda \). Thus, the Palais-Smale condition is established. Note also that
$$
D \mathcal{J}_\lambda(u) = -\Delta u + \lambda e^u(e^u - 1)^{2p+1} + f = \mathcal{F}(u),
$$
where \( \mathcal{F} \) is as defined in Theorem \ref{th2}. Referencing [\cite{MR1196690}, Chapter 2, Theorem 3.2], and in light of (\ref{cqg}), for sufficiently large \( R > 1 \), we have
$$
\operatorname{deg}(\mathcal{F}, B_R, 0) = \operatorname{deg}(D \mathcal{J}_\lambda, B_R, 0) = \sum_{q=0}^{\infty} (-1)^q \operatorname{rank} \mathrm{C}_q(\mathcal{J}_\lambda, u_\lambda) = 1.
$$
This is in contradiction with \( \operatorname{deg}(\mathcal{F}, B_R, 0) = 0 \) as derived from Theorem \ref{th2}. Hence, equation (\ref{meq}) must have at least two distinct solutions. This completes the proof of Theorem \ref{th3} (b).

\vskip 20 pt
\noindent{\bf Acknowledgement}

This work is  partially  supported by  the National Key R and D Program of China 2020YFA0713100 and by the National Natural Science Foundation of China (Grant No. 11721101).

\vskip 20 pt

\bibliographystyle{plain}
\bibliography{D:/Reference/CHS.bib}

\begin{thebibliography}{10}

\bibitem{MR1324400}
Luis~A. Caffarelli and Yi~Song Yang.
\newblock Vortex condensation in the {C}hern-{S}imons {H}iggs model: an
  existence theorem.
\newblock {\em Comm. Math. Phys.}, 168(2):321--336, 1995.

\bibitem{MR1429096}
Dongho Chae and Namkwon Kim.
\newblock Topological multivortex solutions of the self-dual
  {M}axwell-{C}hern-{S}imons-{H}iggs system.
\newblock {\em J. Differential Equations}, 134(1):154--182, 1997.

\bibitem{MR1946331}
Hsungrow Chan, Chun-Chieh Fu, and Chang-Shou Lin.
\newblock Non-topological multi-vortex solutions to the self-dual
  {C}hern-{S}imons-{H}iggs equation.
\newblock {\em Comm. Math. Phys.}, 231(2):189--221, 2002.

\bibitem{MR1196690}
Kung-ching Chang.
\newblock {\em Infinite-dimensional {M}ores theory and multiple solution
  problems}, volume~6 of {\em Progress in Nonlinear Differential Equations and
  their Applications}.
\newblock Birkh\"{a}user Boston, Inc., Boston, MA, 1993.

\bibitem{CHAO2023126787}
Ruixue Chao and Songbo Hou.
\newblock Multiple solutions for a generalized {C}hern-{S}imons equation on
  graphs.
\newblock {\em J. Math. Anal. Appl.}, 519(1):Paper No. 126787, 2023.

\bibitem{chao2022existence}
Ruixue Chao, Songbo Hou, and Jiamin Sun.
\newblock Existence of solutions to a generalized self-dual {C}hern-{S}imons
  system on finite graphs.
\newblock {\em arXiv preprint arXiv:2206.12863}, 2022.

\bibitem{gao2022existence}
Jia Gao and Songbo Hou.
\newblock Existence theorems for a generalized {C}hern-{S}imons equation on
  finite graphs.
\newblock {\em J. Math. Phys.}, 64(9):Paper No. 091502, 12, 2023.

\bibitem{MR3648273}
Huabin Ge.
\newblock Kazdan-{W}arner equation on graph in the negative case.
\newblock {\em J. Math. Anal. Appl.}, 453(2):1022--1027, 2017.

\bibitem{MR3759076}
Huabin Ge and Wenfeng Jiang.
\newblock Yamabe equations on infinite graphs.
\newblock {\em J. Math. Anal. Appl.}, 460(2):885--890, 2018.

\bibitem{grigor2016kazdan}
Alexander Grigor'yan, Yong Lin, and Yunyan Yang.
\newblock Kazdan-{W}arner equation on graph.
\newblock {\em Calc. Var. Partial Differential Equations}, 55(4):Art. 92, 13,
  2016.

\bibitem{MR3542963}
Alexander Grigor'yan, Yong Lin, and Yunyan Yang.
\newblock Yamabe type equations on graphs.
\newblock {\em J. Differential Equations}, 261(9):4924--4943, 2016.

\bibitem{MR3412153}
Boling Guo and Fangfang Li.
\newblock Existence of topological vortices in an {A}belian {C}hern-{S}imons
  model.
\newblock {\em J. Math. Phys.}, 56(10):101505, 10, 2015.

\bibitem{MR2168004}
Jongmin Han and Hee-Seok Nam.
\newblock On the topological multivortex solutions of the self-dual
  {M}axwell-{C}hern-{S}imons gauged {${\rm O}(3)$} sigma model.
\newblock {\em Lett. Math. Phys.}, 73(1):17--31, 2005.

\bibitem{MR4344559}
Xiao~Li Han and Meng~Qiu Shao.
\newblock {$p$}-{L}aplacian equations on locally finite graphs.
\newblock {\em Acta Math. Sin. (Engl. Ser.)}, 37(11):1645--1678, 2021.

\bibitem{MR3033571}
Xiaosen Han.
\newblock The existence of multi-vortices for a generalized self-dual
  {C}hern-{S}imons model.
\newblock {\em Nonlinearity}, 26(3):805--835, 2013.

\bibitem{MR1050529}
Jooyoo Hong, Yoonbai Kim, and Pong~Youl Pac.
\newblock Multivortex solutions of the abelian {C}hern-{S}imons-{H}iggs theory.
\newblock {\em Phys. Rev. Lett.}, 64(19):2230--2233, 1990.

\bibitem{hou2022existence}
Songbo Hou and Jiamin Sun.
\newblock Existence of solutions to {C}hern-{S}imons-{H}iggs equations on
  graphs.
\newblock {\em Calc. Var. Partial Differential Equations}, 61(4):Paper No. 139,
  13, 2022.

\bibitem{hua2023existence}
Bobo Hua, Genggeng Huang, and Jiaxuan Wang.
\newblock The existence of topological solutions to the chern-simons model on
  lattice graphs.
\newblock {\em arXiv preprint arXiv:2310.13905}, 2023.

\bibitem{huaxu2023existence}
Bobo Hua and Wendi Xu.
\newblock The existence of ground state solutions for nonlinear p-laplacian
  equations on lattice graphs.
\newblock {\em arXiv preprint arXiv:2310.08119}, 2023.

\bibitem{huang2020existence}
An~Huang, Yong Lin, and Shing-Tung Yau.
\newblock Existence of solutions to mean field equations on graphs.
\newblock {\em Communications in Mathematical Physics}, 377(1):613--621, 2020.

\bibitem{MR3390935}
Hsin-Yuan Huang, Youngae Lee, and Chang-Shou Lin.
\newblock Uniqueness of topological multi-vortex solutions for a skew-symmetric
  {C}hern-{S}imons system.
\newblock {\em J. Math. Phys.}, 56(4):041501, 12, 2015.

\bibitem{huang2014uniqueness}
Hsin-Yuan Huang and Chang-Shou Lin.
\newblock {U}niqueness of non-topological solutions for the {C}hern-{S}imons
  system with two {H}iggs particles.
\newblock {\em Kodai Mathematical Journal}, 37(2):274--284, 2014.

\bibitem{huang2021mean}
Hsin-Yuan Huang, Jun Wang, and Wen Yang.
\newblock Mean field equation and relativistic {A}belian {C}hern-{S}imons model
  on finite graphs.
\newblock {\em Journal of Functional Analysis}, 281(10):109218, 2021.

\bibitem{MR1050530}
R.~Jackiw and Erick~J. Weinberg.
\newblock Self-dual {C}hern-{S}imons vortices.
\newblock {\em Phys. Rev. Lett.}, 64(19):2234--2237, 1990.

\bibitem{MR3776360}
Matthias Keller and Michael Schwarz.
\newblock The {K}azdan-{W}arner equation on canonically compactifiable graphs.
\newblock {\em Calc. Var. Partial Differential Equations}, 57(2):Paper No. 70,
  18, 2018.

\bibitem{li2023topological}
Jiayu Li, Linlin Sun, and Yunyan Yang.
\newblock Topological degree for chern-simons higgs models on finite graphs.
\newblock {\em arXiv preprint arXiv:2309.12024}, 2023.

\bibitem{MR4490503}
Yang Liu.
\newblock Brouwer degree for mean field equation on graph.
\newblock {\em Bull. Korean Math. Soc.}, 59(5):1305--1315, 2022.

\bibitem{MR4092834}
Shoudong Man.
\newblock On a class of nonlinear {S}chr\"{o}dinger equations on finite graphs.
\newblock {\em Bull. Aust. Math. Soc.}, 101(3):477--487, 2020.

\bibitem{MR4416135}
Linlin Sun and Liuquan Wang.
\newblock Brouwer degree for {K}azdan-{W}arner equations on a connected finite
  graph.
\newblock {\em Adv. Math.}, 404:Paper No. 108422, 29, 2022.

\bibitem{MR1400816}
Gabriella Tarantello.
\newblock Multiple condensate solutions for the {C}hern-{S}imons-{H}iggs
  theory.
\newblock {\em J. Math. Phys.}, 37(8):3769--3796, 1996.

\bibitem{MR1690951}
Guofang Wang and Liqun Zhang.
\newblock Non-topological solutions of the relativistic {${\rm SU}(3)$}
  {C}hern-{S}imons {H}iggs model.
\newblock {\em Comm. Math. Phys.}, 202(3):501--515, 1999.

\bibitem{MR3833747}
Ning Zhang and Liang Zhao.
\newblock Convergence of ground state solutions for nonlinear {S}chr\"{o}dinger
  equations on graphs.
\newblock {\em Sci. China Math.}, 61(8):1481--1494, 2018.

\end{thebibliography}

\end{document}